\documentclass[preprint,1p]{elsarticle}

\makeatletter
 \def\ps@pprintTitle{%
 	\let\@oddhead\@empty
 	\let\@evenhead\@empty
 	\def\@oddfoot{\footnotesize\itshape
 		{} \hfill\today}%
 	\let\@evenfoot\@oddfoot
 }
\makeatother
\usepackage[unicode]{hyperref}

\usepackage{latexsym}
\usepackage{indentfirst}
\usepackage{amsxtra}
\usepackage{amssymb}
\usepackage{amsthm}
\usepackage{amsmath}
\usepackage{mathrsfs} 

\usepackage{xcolor}
\usepackage{color}

\usepackage{amsfonts}

\usepackage[capitalise]{cleveref}

\newtheorem{theor}{Theorem}
\newtheorem*{theor*}{Theorem}
\newtheorem{prop}[theor]{Proposition}
\newtheorem{lemma}[theor]{Lemma}
\newtheorem{cor}[theor]{Corollary}
\newtheorem*{cor*}{Corollary}
\theoremstyle{definition}               
\newtheorem{defin}[theor]{Definition}
\newtheorem{ex}{Example}
\newtheorem{exs}[ex]{Examples}
\newtheorem{rem}{Remark}

\DeclareMathOperator{\Aut}{Aut}
\DeclareMathOperator{\End}{End}
\DeclareMathOperator{\id}{id}
\DeclareMathOperator{\im}{Im}

\DeclareMathOperator{\E}{E}
\DeclareMathOperator{\Z}{Z}

\usepackage[makeroom]{cancel}
 
\usepackage{soul}
\usepackage{tikz}

\begin{document}

\begin{frontmatter}
	\title{Idempotent set-theoretical solutions of the pentagon equation\tnoteref{mytitlenote}}
	\tnotetext[mytitlenote]{This work was partially supported by the Dipartimento di Matematica e Fisica ``Ennio De Giorgi'' - Università del Salento. The author is member of GNSAGA (INdAM) and of the non-profit association ADV-AGTA.}
		\author[unile]{Marzia~MAZZOTTA}
	\ead{marzia.mazzotta@unisalento.it}
	\address[unile]{Dipartimento di Matematica e Fisica ``Ennio De Giorgi''
		\\
		Universit\`{a} del Salento\\
		Via Provinciale Lecce-Arnesano \\
		73100 Lecce (Italy)\\}

\begin{abstract} 
A \emph{set-theoretical solution of the pentagon equation} on a non-empty set $X$ is a function $s:X\times X\to X\times X$ satisfying the relation $s_{23}\, s_{13}\, s_{12}=s_{12}\, s_{23}$, with $s_{12}=s\times  \,\id_X$,  $s_{23}=\id_X \times \, s$ and  $s_{13}=(\id_X\times \, \tau)s_{12}(\id_X\times \,\tau)$, where $\tau:X\times X\to X\times X$ is the flip map given by $\tau(x,y)=(y,x)$, for all $x,y\in X$. Writing a solution as $s(x,y)=(xy ,\theta_x(y))$, where $\theta_x: X \to X$ is a map, for every $x\in X$, one has that $X$ is a semigroup.\\ 
In this paper, we study idempotent solutions, i.e., $s^2=s$, by showing that the idempotents of $X$ have a crucial role in such an investigation. In particular, we describe all such solutions on monoids having central idempotents.
Moreover, we focus on idempotent solutions defined on monoids for which the map $\theta_1$ is a monoid homomorphism. 
\end{abstract}

\end{frontmatter}

\section*{Introduction}

	 If $V$ is a vector space over a field $F$, a linear map $S: V \otimes V \to V \otimes V$ is said to be a \emph{solution of the pentagon equation} on $V$ if it satisfies the relation
	\begin{align}\label{vec_equa}
	S_{12}S_{13}S_{23}=S_{23}S_{12},
	\end{align}
	where 
		$S_{12}=S\otimes \id_V$,\,$S_{23}=\id_V\otimes \, S$,\,$S_{13}=(\id_V\otimes\, \Sigma)\,S_{12}\;(\id_V\otimes \, \Sigma)$,
	with $\Sigma$ the flip operator on $V\otimes V$, i.e., $\Sigma(u\otimes v)=v\otimes u$, for all $u,v\in V$.
The pentagon equation classically originates from the field of Mathematical Physics, but it has several applications and appears in different areas of mathematics, also with different terminologies (see, for instance, \cite{Za92, BaSk93, Wo96, BaSk98, Mi98, JiLi05, Kawa10, Ka11}). 
To know more about contexts in which the pentagon equation appears, we refer to the introduction of the paper by Dimakis and M\"{u}ller-Hoissen \cite{DiMu15}  along with the references therein.  \\
	In 1998, Kashaev and Sergeev \cite{KaSe98} explicitly first highlighted an existing link between solutions on a vector space viewed as the space $F^X$ of all functions from a finite set $X$ to $F$ and maps from $X \times X$ into itself satisfying a certain relation. Into the specific, to any map $s:X\times X\to X\times X$ one can associate a linear operator $S:F^{X \times X} \to F^{X \times X}$ defined as $
	S(f)(x,y)= f(s(x,y))$, for all $x, y \in X$.
	If the map $s$ satisfies the ``reversed" pentagon relation 
	\begin{equation}\label{set_equa}
	s_{23}s_{13}s_{12}=s_{12}s_{23},
	\end{equation}
	where
		$s_{12}=s\times \id_X$,  $s_{23}=\id_X\times s$, $s_{13}=(\id_X\times \tau)\,s_{12}\,(\id_X\times \tau)$, with $\tau(x,y)=(y,x)$, for all $x,y\in X$,
	then the linear map $S$ is a solution of the pentagon equation on $F^X$.  We call the map $s$ as above \emph{set-theoretical solution of the pentagon equation}, or briefly \emph{solution}, on $X$. In the pioneering paper \cite{KaSe98},  one can also find the first systematic way to obtain solutions on closed under multiplications subsets of arbitrary groups.   However, set-theoretical-type solutions had already appeared before in the form of \emph{pentagonal transformations} on differential manifolds and on measured spaces in two papers: those by Zakrzewski \cite{Za92} 
and Baaj and Skandalis \cite{BaSk98}, respectively. One can transcribe these solutions in purely algebraic terms e note that they are the first instances of bijective solutions.  \\
	Attention only to set-theoretical solutions has been recently given in \cite{CaMaMi19}. Following the notation therein, writing a solution $s: X \times X \to X \times X$ as $s(x,y)=(xy, \theta_x(y))$, where $\theta_x$ is a map from $X$ into itself, for every $x \in X$,  one has that $X$ is a  semigroup and 
	\begin{align*}\label{p_one}
	\theta_x(y)  \theta_{x y}(z)&=\theta_x(y z ),\tag{P1}\\
	\label{p_two}\theta_{\theta_x(y)}\theta_{x\ y}&=\theta_y, \tag{P2}
	\end{align*}
	for all $x,y,z \in X$. In \cite[Theorem 15]{CaMaMi19}, a complete description of all solutions defined on a group is given. In this case, it holds that $\theta_x(y)=\theta_1(x)^{-1}\theta_1(xy)$, for all $x, y\in X$, where $1$ is the identity of the group $X$, and it is sufficient to study the set $\ker\theta_1=\{x \in X \, \mid \, \theta_1(x)=1\}$ that it a normal subgroup of $X$, even if, in general, $\theta_1$ is not a homomorphism.  	However, describing all solutions on arbitrary semigroups seems very difficult since there are many even in the case of small-order semigroups. For instance, it is sufficient to look \cite[Appendix B]{tesi}, where, as suggested by Rump, all the $202$ non-isomorphic solutions on semigroups of order $3$ have been computed. \\
A first step could be studying specific classes of solutions. In this regard,  a characterization of all involutive solutions, i.e., $s^2=\id_{X \times X}$, has been provided by Colazzo, Jespers, and Kubat in their recent paper \cite{CoJeKu20}. Moreover, in \cite{CaMaSt20}, Catino, Mazzotta, and Stefanelli described the class of solutions that satisfy both the pentagon equation and the quantum Yang-Baxter equation (see \cite{Dr92}). Indeed, as one can note, the pentagon equation is the quantum Yang-Baxter equation with the middle term missing on the right-hand side. An easy example of a map that satisfies both equations is given by $s(x, y) = (f(x), g(y))$, with $f, g$ idempotent maps from a set
$X$ into itself such that $fg = gf$. In this last paper, one can also find some methods to construct solutions of the pentagon equation such as, for instance, on the \emph{matched product} $S \bowtie T$ of two
semigroups $S$ and $T$ (see \cite[Definition 1]{CaMaSt20}), namely a new semigroup on the cartesian product of $S$ and $T$ including the \emph{classical Zappa product} of $S$ and $T$ (see \cite[Definition 1.1]{Ku83}).\\
Recently, idempotent left non-degenerate solutions of the Yang-Baxter equation have been completely described  (see \cite{Lebed17, StanVoj21, COJeKuAntVe22x} and the classifications therein). A similar study can also be done for the pentagon equation, motivated by the fact that already for semigroups of small order there are several of idempotent type, i.e., $s^2=s$. The question of studying this class of solutions arose explicitly during the talk held in the occasion of ``\emph{First UMI meeting of PhD students}" \cite{Maz100Umi}.

	In this paper, we deal with \emph{idempotent set-theoretical solutions of the pentagon equation}. One can easily check that a solution $s(x,y)=(xy, \theta_x(y))$ on a semigroup $X$ is idempotent if and only if
  \begin{align}\label{i_one}
     &x y  \theta_x(y)=x y,\tag{I1}\\
   \label{i_two}  &\theta_{x y}\theta_x(y)=\theta_x(y), \tag{I2}
 \end{align}
for all $x, y\in X$. We show that the idempotents of $X$ have a crucial role in finding them and, for that, we exhibit some useful properties of the maps $\theta_x$ involving the idempotents. In particular, we focus on solutions defined on monoids, by showing that, unlike solutions defined on groups,  it is not possible to find a way to write the maps $\theta_x$ by means of the map $\theta_1$, where $1$ is the identity of the monoid. We provide a description theorem for idempotent solutions on monoids $M$ having central idempotents, namely, $\E(M) \subseteq \Z(M)$, by illustrating that it is enough to construct specific idempotent maps $\theta_e$, for every $e \in \E(M)$.  Furthermore, in this situation, the map $\theta_1$ is an idempotent monoid homomorphism, and, for this reason, we deepen idempotent solutions on monoids  satisfying this additional property. In this case, all the maps $\theta_x$ have to be derived considering the kernel congruence of the function $\theta_1$, namely, the set $\ker\theta_1=\{(x,y) \in M \times M\, \mid \, \theta_1(x)=\theta_1(y)\}$ (see \cite{Rhod89} for more details). Indeed, $(\theta_x(y), y) \in \ker \theta_1$, for all $x, y \in M$, and $\theta_1(M)$ is a system of representatives of $M/ \ker \theta_1$. \\
Finally, we collect some properties of idempotent solutions on arbitrary semigroups, that could be useful in a future study of these maps in the more general case.

\bigskip

\section{Basics on solutions}
\noindent In this section, we give some basics on the solutions. Moreover, we introduce some classes of solutions and provide several examples.

\medskip

From now on, following the notation used in \cite[Proposition 8]{CaMaMi19}, given a semigroup $X$, we will briefly call \emph{solution} on $X$ any map $s: X \times X \to X \times X$ given by $s(x,y)=(x  y, \theta_x(y))$, where $\theta_x$ is a map from $X$ into itself, satisfying \eqref{p_one} and \eqref{p_two}.

\medskip

\begin{ex}\label{ex_militaru} (cf. \cite{Mi98}) Let $X$ be a set and $f, g: X \to X$ idempotent maps such that $fg=gf$. Set $x \cdot y=f(x)$, for all $x, y \in X$, one has that $(X, \cdot)$ is a semigroup and the map 
	$s(x,y)=\left(x\cdot y,\, g\left(y\right)\right)$ is a solution on $X$.
\end{ex}
Note that the previous example belongs to the class of P-QYBE solutions, namely the maps that are solutions both to the pentagon and the Yang-Baxter equations \cite{CaMaSt20}.

\medskip

\begin{defin}
    Let $(X, \cdot)$ and $(Y, \ast)$ be two semigroups and $s(x,y)=(x \cdot y, \theta_x(y))$ and $r(a, b)=(a  \ast b, \eta_a(b))$ two solutions on $X$ and $Y$, respectively. Then, $s$ and $r$ are \emph{isomorphic} if there exists a semigroup isomorphism $f: X \to Y$ such that
    \begin{align}\label{isomor}
        f\theta_x(y)=\eta_{f\left(x\right)}f(y),
    \end{align}
    for all $x, y \in X$, or, equivalently, $(f \times f)s=r(f \times f)$.    \end{defin}

A complete description  in the case of a group is given in \cite[Theorem 15]{CaMaMi19}. For the sake of completeness, we recall such a result below.
	\begin{theor}\label{teo_gruppi}
	Let $G$ be a group. Consider a normal subgroup $K$ of $G$ and a system of representatives $R$ of $G/K$ such that $1 \in R$.  If $\mu: G \to R$ is a map such that $\mu(x) \in Kx$, for every $x \in G$, then the map $
	s(x,y)=\left(xy, \mu\left(x\right)^{-1}\mu\left(xy\right)\right),
$
	for all $x,y \in G$, is a solution on $G$. \\
	Conversely, if $s$ is a solution on $G$, then the set
	$K=\{x \in G\, \mid\, \theta_1(x)=1\}$
	is a normal subgroup of $G$ for which $\im \theta_1$ is a system of representatives of $G/K$ that contains $1$, $\theta_1(x) \in Kx$, for every $x\in G$, and $
	s(x,y)=\left(xy, \ \theta_1\left(x\right)^{-1}\theta_1\left(xy\right)\right),$
	for all $x, y \in G$.
\end{theor}
\noindent By \cref{teo_gruppi} and making explicit the condition \eqref{isomor}, it is easy to check that two solutions $s(x,y)=(x y, \theta_x(y))$ and $r(x,y)=\left(xy, \eta_x\left(y\right)\right)$ on the same group $G$ are isomorphic via $f \in \Aut(G)$ if and only if  $f\theta_1=\eta_1f$, i.e., $f$ sends the system of representatives $\theta_1(G)$ into the other one $\eta_1f(G)$.

\medskip

	However, describing all the solutions, up to isomorphisms, on arbitrary semigroups turns out to be very hard. Indeed, even in the case of semigroups of small order, there are a lot of solutions, as one can see in \cite[Appendix B]{tesi}. A first step could be studying specific classes of solutions. In this regard, one can find  in \cite[Theorem 5.5]{CoJeKu20} a complete description of all \emph{involutive} solutions.

	\begin{defin}
	    Let $X$ be a semigroup and $s(x,y)=(xy, \theta_x(y))$ a solution on $X$. We say that the map $s$ is
	    \begin{enumerate}
	        \item[-] \emph{non-degenerate} if $\theta_x$ is bijective, for every $x \in X$;
	        \item[-]\emph{involutive} if $s^2=\id_{X\times X}$;
	        \item[-] \emph{idempotent} if $s^2=s$.
	    \end{enumerate}
	\end{defin}

\medskip

\begin{exs}\hspace{1mm}\label{exs_PE}
\begin{enumerate}
    \item \cite[Examples 2-2.]{CaMaMi19}\, Let $X$ be a semigroup and $\gamma :X \to X$ a map.  Then, the map 
	$s(x,y)=\left(xy,\gamma\left(y\right)\right),$
	for all $x,y \in X$, is a solution if and only if $\gamma\in \End(X)$ and $\gamma^2=\gamma$. One can easily check that such a solution $s$ is non-degenerate if and only if $\gamma=\id_X$.  
	\vspace{1mm}
	\item As a particular case of $1.$, if $X$ is a semigroup and $e \in \E(X)$, where $\E(X)$ denotes the set of the idempotents of $X$,  the map
	$ s(x,y)=\left(xy,e\right),$
	for all $x,y \in X$, is an idempotent solution.
		\vspace{1mm}
	\item 
    Let $X$ be a semigroup belonging to the variety $\mathcal{S}:= [abc = bc]$ (see \cite[p. 370]{Mo03}).
   Then, the map  
    $s\left(x,y\right)= \left(xy, xy\right),$ for all $x,y \in X$, is an idempotent solution. 
	\vspace{1mm}
	\item Every Clifford semigroup $X$ gives rise to the idempotent solution $s$ given by
	$s(x,y)=\left(xy, y^{-1}y\right),$
	for all $x, y \in X$.
	Recall that a Clifford semigroup $X$ is a semigroup in which every $x\in X$ admits a unique $x^{-1}\in X$ such that $xx^{-1}x=x, x^{-1}xx^{-1}=x^{-1}$, and $xx^{-1}=x^{-1}x$ (see \cite[Exercise II.2.14]{Pe84}). \vspace{1mm}
	\item \cite[Appendix B]{tesi} Let $X=\{0,a,b\}$ and $S$ the null semigroup on $X$, i.e., $xy=0$, for all $x, y \in X$. Consider the maps $\theta_0=\id_S$ and $\theta_a=\theta_b$ such that $\theta_a(0)=0$, $\theta_a(a)=b$, and $\theta_a(b)=a$. Then, the map $s(x, y)=(0, \theta_x(y))$ is an idempotent and non-degenerate solution on $S$.
\end{enumerate}
\end{exs}

\medskip

Other classes of solutions that can be studied are the commutative and the cocommutative ones (see \cite[Definition 6]{CaMaMi19}). These kinds of solutions are in analogy to the commutative and the cocommutative multiplicative unitary operators, i.e., solutions of \eqref{vec_equa} defined on Hilbert spaces (see \cite[Definition 2.1]{BaSk93}).

\begin{defin}
	A solution $s: X \times X \to X \times X$ is said to be 
	\begin{enumerate}
	    \item[-] \emph{commutative} if $s_{12}s_{13}=s_{13}s_{12}$; \item[-] \emph{cocommutative} if $s_{13}s_{23}=s_{23}s_{13}$.
	\end{enumerate}
\end{defin}

\noindent If $X$ is a  semigroup and $s(x,y)=(xy, \theta_x(y))$ a solution on $X$, it is a routine computation to check that the map $s$ is commutative if and only if 
\begin{align*}
  xzy=xyz \quad \text{(C1)} \qquad \qquad\qquad \theta_x=\theta_{xy} \quad\text{(C2)}
\end{align*}
for all $x,y,z\in X$. 
Instead, $s$ is cocommutative if and only if 
\begin{align*}
    x\theta_y(z)=xz \quad \text{(CC1)} \qquad \qquad \theta_x\theta_{y}=\theta_y\theta_x \quad\text{(CC2)}
\end{align*}
for all $x,y,z \in X$.\\
There exist solutions that are both commutative and cocommutative, such as the maps in \cref{ex_militaru}.  Moreover, according to \cite[Corollary 3.4]{CoJeKu20}, if $s$ is an involutive solution, then $s$ is both commutative and cocommutative. 

\medskip

If $M$ is a monoid, it follows by (CC1) that the unique cocommutative solution on $M$ is given by  $s(x,y)=(xy,y)$, for  all $x,y \in M$. In the next result, we describe all the commutative solutions on a monoid.

\begin{prop}
Let $M$ be a monoid. Then, a solution $s(x,y)=(xy, \theta_x(y))$ on $M$ is commutative if and only if $M$ is a commutative monoid and $\theta_x=\gamma$, with $\gamma \in \End(M)$, $\gamma^2=\gamma$, for every $x \in M$.
\begin{proof}
By \cref{exs_PE}-1., $s(x,y)=\left(xy, \gamma\left(y\right)\right)$, is a commutative solution on $M$. Conversely, let us assume that $s(x,y)=(xy, \theta_x(y))$ is commutative. Then, by substituting $x=1$ in (C1), the monoid $M$ is commutative and, by (C2), $\theta_1=\theta_y$, for every $y \in M$. Thus, by \eqref{p_one} and \eqref{p_two} the claim follows.
\end{proof}
\end{prop}

\bigskip

\section{Properties of the maps \texorpdfstring{$\theta_x$}{} involving the idempotents}
\noindent In this section, we provide some properties of the maps $\theta_x$ which involve the idempotents of arbitrary semigroups and that will be used in the next.

\medskip

Firstly, according to \cite[p. 69]{Ho95}, among the idempotents in any semigroup $X$, there is a natural partial order relation by the rule that
\begin{align*}
    \forall\, e,f \in \E(X) \quad e \leq f \Longleftrightarrow ef=fe=e.
\end{align*}
Thus, we can collect the following easy properties for the maps $\theta_e$ on $X$.
\begin{lemma}\label{lemma_idemp}
 Let $X$ be a semigroup, $e, f \in \E(X)$ such that $e \leq f$,  and $s(x,y)=(xy, \theta_x(y))$ a solution on $X$. Then, the following hold:
 \begin{enumerate}
     \item $\theta_e(f) \in \E(X)$,
     \item $\theta_e(e) \leq \theta_e(f)$,
     \item $\theta_f=\theta_{\theta_e(f)}\theta_e$.
 \end{enumerate}
\end{lemma}

\medskip

Now, following \cite[p. 22]{ClPre61}, given a semigroup $X$ and $e \in \E(X)$, then $e$ is a \emph{left identity} (or \emph{right identity}) if $ex=x$ (or $xe=x$), for every $x \in X$, and the sets
\begin{align*}
    eX&=\{x \in X \, \mid \, ex=x\}\qquad Xe=\{x \in X \, \mid \, xe=x\}
\end{align*}
are respectively the principal right and left ideals of $X$ generated by $e$. Moreover, we set $eXe=eX \cap Xe$. We can check the following properties.
\begin{lemma}\label{lemma_es_se}
Let $X$ be a semigroup, $e \in \E(X)$, and $s(x, y)=(xy, \theta_x(y))$ a solution on $X$. 
\begin{enumerate}
    \item If $x \in Xe$, then $\theta_x(e) \in \E(X)$,

 \item   if $x \in eX$, then
    $\theta_e(x) \in \theta_e(e)X$,
\end{enumerate}
\begin{proof}
Initially, assume that $x \in Xe$. Then, by \eqref{p_one}, $\theta_x(e)=\theta_x(e)\theta_{xe}(e)=\theta_x(e)\theta_x(e)$. \\
Now, assume that $x \in eX$. Thus, using \eqref{p_one}, we have that $\theta_e(x)=\theta_e(ex)=\theta_e(e)\theta_e(x)$. Hence, since by \cref{lemma_idemp}-$1.$ $\theta_e(e) \in \E(X)$, we get $\theta_e(x) \in \theta_e(e) X$.
\end{proof}
\end{lemma}

\medskip

\noindent As a direct consequence of the previous lemma, we have the following properties for arbitrary solutions defined on a monoid.
\begin{lemma}\label{prop_monoid}
Let $M$ be a monoid with identity $1$ and $s(x,y)=(xy, \theta_x(y))$ a solution on $M$. Then, the following hold:
\begin{enumerate}
    \item $\theta_x(1) \in \E(M)$,
    \item $\theta_1=\theta_{\theta_x(1)}\theta_x$,
    \item $\theta_1(x) \in \theta_1(1)M$,
    \item $\theta_x=\theta_{\theta_1(x)}\theta_x$,   
    \end{enumerate}
    for every $x \in M$.
\end{lemma}

In particular, it follows by \cref{prop_monoid}-$4.$ that the only non-degenerate solution on a monoid $M$ is that for which $\theta_x=\id_M$, for every $x \in M$.

\bigskip

We conclude this section focusing on solutions on semigroups having central idempotents, i.e., it holds $xe=ex$, for all $e \in \E(X)$ and $x \in X$. Obviously, in this case, $Xe=eX$. The result we provide is consistent with  Lemma 11 and the equation (4) of \cite{CaMaMi19}. Let us first recall, that if $e \in \E(X)$, the set
\begin{align*}
    H_e=\{x \in Xe \, \mid \, \exists \, y \in Xe \quad xy=yx=e\}
\end{align*}
is a group with identity $e$. In particular, if $e$ and $f$ are distinct idempotents, then $H_e$ and $H_f$ are disjoint. If $x \in H_e$, let us denote by $x^{-}$ the inverse of $x$ in $H_e$.

\begin{prop}\label{prop_ef_thetax}
Let $X$ be a semigroup having central idempotents, $e \in \E(X)$, $x \in H_e$, and $s(x,y)=(xy, \theta_x(y))$ a solution on $X$. Then, the following hold:
\begin{enumerate}
\item $\theta_e(e) \leq \theta_x(e)$;
 \item  $\theta_e(x) \in H_{\theta_e(e)}$ and in particular $\theta_e(x)^{-}=\theta_x\left(x^{-}\right)$;
 \item if $f \in \E(X)$ is such that $f \leq e$ and $y\in H_f$, then $
    \theta_x(y)=\theta_e\left(x\right)^{-}\theta_e(xy).$
 \end{enumerate}
\begin{proof}
At first, we have that, by \cref{lemma_idemp}-$1.$, $\theta_e(e) \in \E(X)$, and, by \cref{lemma_es_se}-$a.$, it holds that $\theta_x(e) \in \E(X)$. Thus,
\begin{align*}
   \theta_e(e) \theta_x(e)&=\theta_e\left(xx^{-}\right)\theta_x(e)=\theta_e(x)\theta_{ex}\left(x^-\right)\theta_x(e)=\theta_e(x)\theta_{ex}(e)\theta_{ex}\left(x^-\right) \\
   &=\theta_e(xe)\theta_{ex}\left(x^-\right)=\theta_e(x)\theta_{ex}\left(x^-\right)=\theta_{e}\left(xx^-\right)=\theta_e(e),
\end{align*}
and so $1.$ follows. Besides, by \cref{lemma_es_se}-$c.$, $\theta_e(x) \in X\theta_e(e)$ and also $\theta_x\left(x^-\right) \in X\theta_e(e)$, since $\theta_x\left(x^-\right)\theta_e(e)=\theta_x\left(x^-e\right)=\theta_x\left(x^-\right)$. Hence, we get
\begin{align*}
   &\theta_e(x)\theta_{x}\left(x^{-}\right)= \theta_e(x)\theta_{ex}\left(x^{-}\right)=\theta_e\left(xx^{-}\right)=\theta_e(e)\end{align*}
   and
\begin{align*}   \theta_{x}\left(x^{-}\right)\theta_e(x)&=\theta_{x}\left(x^{-}\right)\theta_e(x)\theta_e(e) &\mbox{$\theta_x\left(x^-\right) \in X\theta_e(e)$}\\
&=\theta_{x}\left(x^{-}x\right)\theta_e(e)\\
&=\theta_x(e)\theta_e(e)\\
&=\theta_e(e).
\end{align*}
Finally, if $f \in \E(X)$ is such that $f \leq e$ and $y \in H_f$, then, by $2.$, we obtain
\begin{align*}
    \theta_x(y)=\theta_x\left(efy\right)=\theta_x\left(xx^{-}fy\right)=\theta_x\left(x^{-}\right)\theta_{e}(xy)=\theta_e\left(x\right)^{-}\theta_e(xy),
\end{align*}
which completes the proof.
\end{proof}
\end{prop}

\bigskip

\section{Properties of the maps \texorpdfstring{$\theta_x$}{}  of idempotent solutions}
\noindent In this section, we collect some properties of the maps $\theta_x$ of  idempotent solutions on arbitrary semigroups.

\medskip

Initially, it is a routine computation to check that a solution $s(x,y)=(xy, \theta_x(y))$ on a semigroup $X$ is idempotent if and only if 
 \begin{align}
     &xy\theta_x(y)=xy,\tag{I1}\\
     &\theta_{xy}\theta_x(y)=\theta_x(y), \tag{I2}
 \end{align}
for all $x,y \in X$. In particular, by \eqref{i_one}, if $X$ is a group the unique idempotent solution on $X$ is the map $s(x,y)=(xy, 1)$; such a solution $s$ belongs to the class of solutions discussed in \cref{exs_PE}-$1.$. On the other hand, considering this class of solutions on monoids, one can easily check the following result.
\begin{prop}\label{prop_gamma}
Let $M$ be a monoid and $\gamma: M \to M$ a map. Then, $s(x,y)=(xy, \gamma(y))$ is an idempotent solution on $M$  if and only if $\gamma\in \End(M)$, $\gamma^2=\gamma$, and $x\gamma(x)=x$, for every $x \in M$. \end{prop}
\noindent By the previous proposition, in particular, the solution $s(x,y)=(xy, y)$ on $M$ is idempotent if and only if $M$ is an idempotent monoid. 

\medskip

However, taking monoids even of small orders, one can note that among the solutions there are several of the idempotent type that do not belong to the class of solutions in \cref{exs_PE}-$1.$. The following is an easy example.

\begin{ex}\cite[Appendix B]{tesi}
Let $X=\{1,a,b\}$ and $M$ be the commutative monoid on $X$ with identity $1$ and multiplication given by $a^2=a, ab=a,  b^2=1.$ Then, there are three idempotent solutions up to isomorphism:
\begin{enumerate}
    \item $s(x, y)=(xy, 1)$;
    \item $r(x,y)=(xy, \gamma(y))$, with  $\gamma :M \to M$ defined by $\gamma(1)=\gamma(b)=1$ and $\gamma(a)=a$;
    \item $t(x,y)=(xy, \theta_x(y))$, with $\theta_x:M \to M$ the map given by $\theta_x(1)=1, \, \theta_x(a)=a$, for every $x \in X$,  and $\theta_1(b)=\theta_b(b)=1$ and $\theta_a(b)=b$.
\end{enumerate}
\end{ex}

\medskip

Based on the above arguments, in the next, we will focus on idempotent solutions. We first give the following properties, which hold, in general, for idempotent solutions on arbitrary semigroups and whose proof is a routine computation.

\begin{prop} \label{prop_sidemp_semigruppo}
Let $X$ be a semigroup and $s(x,y)=(xy, \theta_x(y))$ an idempotent solution on $X$. Then, the following hold:
\begin{enumerate}
    \item $\theta_{\theta_x(y)}=\theta_y,$
    \item $\theta_y=\theta_y\theta_{xy},$
    \item $\theta_x(yz)=\theta_x(yz)\theta_y(z)$,
\end{enumerate}
for all $x, y,z \in X$.
\end{prop}

\medskip

\begin{prop}\label{prop_Se}
Let $X$ be a semigroup, $e \in \E(X)$, and $s(x,y)=(xy, \theta_x(y))$ an idempotent solution on $X$.
\begin{enumerate} 
    \item If $x \in Xe$, then
    \begin{enumerate}
    \item[a.] $x \in X\theta_x(e)$,
    \item[b.]  $\forall y \in X \quad \theta_y(x) \in X\theta_x(e) $,
    \item[c.] $\theta_e=\theta_e\theta_x$.
\end{enumerate}
\item If $x \in eS$, then
\begin{enumerate}
    \item[d.]  $\theta_e(x) \in \E(X)$, 
    \item[e.]   $x \in X \, \theta_e(x)$,
     \item[f.] $\forall y \in X \quad \theta_y(x) \in X \theta_e(x)$,
    \item[g.] $\theta_x$ is an idempotent map.
\end{enumerate}
\end{enumerate}

\begin{proof}
Initially, assume that  $x \in Xe$. Then, by \cref{lemma_es_se}-$a.$, $\theta_x(e) \in \E(X)$. Moreover, by \eqref{i_one}, we get $x=xe=xe\theta_x(e)=x\theta_x(e)$. Besides, if $y \in X$, by  \cref{prop_sidemp_semigruppo}-$3.$, we have that \begin{align*}
    \theta_y(x)=\theta_y(xe)=\theta_y(xe)\theta_x(e)=\theta_y(x)\theta_x(e),
\end{align*}
and so $b.$ follows. Finally, by \cref{prop_sidemp_semigruppo}-$2.$ , $\theta_e=\theta_e\theta_{xe}=\theta_e\theta_x$, i.e., $c.$ holds.\\
Now, suppose that $x \in eX$. At first, by \eqref{i_one}, we obtain that $x=ex=ex\theta_e(x)=x\theta_e(x)$ and so\begin{align*}
    \theta_e(x)=\theta_e\left(x\theta_e\left(x\right)\right)=\theta_e(x)\theta_{ex}\theta_e(x)=\theta_e(x)^2,
\end{align*}
where in the last equality we apply \eqref{i_two}. Thus, $d.$ and $e.$ follow. Moreover,
\begin{align*}
    \theta_x&=\theta_{\theta_e(x)} &\mbox{by \cref{prop_sidemp_semigruppo}-$1.$}\\
    &=\theta_{\theta_{ex}\theta_e(x)}\theta_{ex\theta_e(x)} &\mbox{by \eqref{p_two}} \\
    &=\theta_{\theta_e(x)}\theta_x &\mbox{by \eqref{i_two}-\eqref{i_one}}\\
    &=\theta_x\theta_x &\mbox{by \cref{prop_sidemp_semigruppo}-$1.$}
\end{align*}
hence $\theta_x$ is an idempotent map. Finally, if $y \in X$, by \cref{prop_sidemp_semigruppo}-$3.$ , we have that \begin{align*}
    \theta_y(x)=\theta_y(ex)=\theta_y(ex)\theta_e(x)=\theta_y(x)\theta_e(x).
\end{align*}
Therefore, the claim follows.
\end{proof}
\end{prop}

\medskip

\noindent As a consequence of the previous proposition, we obtain the following result.
\begin{cor}\label{cor_eSe}
Let $X$ be a semigroup, $e \in \E(X)$, $x \in eXe$, and $s(x,y)=(xy, \theta_x(y))$ an idempotent solution on $X$. Then, the following hold:
\begin{enumerate}
\item $\theta_e(x) \in \E(X)$,
    \item $x \in X\theta_e(x) \cap X\theta_x(e)$,
    \item $\theta_e=\theta_e\theta_x$,
    \item $\forall y \in X \quad \theta_y(x) \in X \theta_e(x)  \cap X\theta_x(e)$,
    \item $\theta_x$ is idempotent.
\end{enumerate}
\end{cor}

\medskip

\begin{rem}\label{rem_theta1_idem}
 Note that if $s$ is an idempotent solution on a monoid $M$ with identity $1$, in general, $\theta_1(\E(M)) \neq \E(M)$. Indeed, if we consider the set $X=\{1,a,b\}$ and the commutative monoid $M$ on $X$ with identity $1$, $\E(M)=\{1, a\}$, and such that $ab=a$, we have that the only idempotent solution is $s(x,y)=(xy, 1).$
\end{rem}

\bigskip

\section{Idempotent solutions on monoids having central idempotents}
\noindent In this section, we focus only on idempotent solutions defined on monoids. In particular, we will give a description theorem for idempotent solutions defined on monoids having central idempotents. In addition, we will show that specific idempotent solutions are strictly linked to the kernel congruence of an idempotent monoid homomorphism.
\medskip

Initially, given an idempotent solution $s$ on $M$, all the statements of \cref{cor_eSe} hold for the identity $1$. In particular, $\theta_1(x) \in \E(M)$ and $x \in M \theta_1(x)$, for every $x \in M$. As a consequence, we have the following:

\begin{prop}
Let $M$ be a cancellative monoid. Then, $s(x,y)=(xy, 1)$ is the unique idempotent solution on $M$.
\begin{proof}
Since the identity is the unique idempotent of $M$, then $\theta_1(x)=1$, for every $x \in M$. Moreover, by \cref{prop_sidemp_semigruppo}-$1.$, $\theta_x=\theta_{\theta_1(x)}$, for every $x \in M$, and so the claim follows.
\end{proof}
\end{prop}

Thus, from now on we will consider not cancellative monoids. We have the following properties.

\begin{prop}\label{prop_sidemp_monoide}
If $M$ is a monoid and $s(x,y)=(xy, \theta_x(y))$ an idempotent solution on $M$, then
\begin{enumerate}
    \item  $\theta_1(1)=1$,
\item $\theta_x=\theta_{\theta_1(x)}$, 
    \item  $\theta_1(x) \leq \theta_x(1)$, 
    \item $\theta_x$ is idempotent, 
\end{enumerate}
for every $x \in M$.
\begin{proof}
The first statement directly follows by \eqref{i_one}. The second one follows by \cref{prop_sidemp_semigruppo}-$1.$ and the fourth one by $g.$ in \cref{prop_Se}. Moreover, if $x \in M$, then $\theta_1(x)=\theta_1(x1)=\theta_1(x)\theta_x(1)$. On the other hand, by \eqref{i_two}, it holds $\theta_x\theta_1(x)=\theta_1(x)$, and so we obtain
\begin{align*}
    \theta_x(1)\theta_1(x)&=\theta_x(1)\theta_x\theta_1(x)=\theta_x(1\theta_1(x))=\theta_x\theta_1(x)=\theta_1(x),
\end{align*}
i.e., $\theta_1(x) \leq \theta_x(1)$.
\end{proof}
\end{prop}

\medskip

Given a monoid $M$, recall that a \emph{right unit} is an element $r$ of $M$ for which there exists $r'\in M$ such that $rr'=1$.
Analogously, $l\in M$ is a \emph{left unit} of $M$ if there exists a left inverse $l'\in M$ such that $l'l=1$. Next, we prove some properties that hold for any right unit of a monoid $M$ and that can be shown also for any left unit $l \in M$ (exchanging the roles of $r$ and $l'$ and of $r'$ and $l$, respectively).

\begin{prop}\label{prop_right}
Let $M$ be a monoid and $s(x,y)=(xy, \theta_x(y))$ an idempotent solution on $M$. If $r\in M$ is a right unit of $M$, then the following hold:
\begin{enumerate}
    \item $\theta_r\left(r'\right)=1$, where $r'\in M$ is such that $rr'=1$,
    \item $\theta_1(r)=1$,
    \item $\theta_1=\theta_r$.
\end{enumerate}
\begin{proof}
The equality $\theta_r\left(r'\right)=1$ follows by setting $x=r$ and $y=r'$ in \eqref{i_one}.
As a consequence, we have that
$\theta_1(r)=\theta_1(r)\cdot 1=\theta_1(r)\theta_r\left(r'\right)=\theta_1\left(rr'\right)=\theta_1(1)=1$.  Moreover,  by  \cref{prop_monoid}-$4.$, it follows that
$\theta_r=\theta_{\theta_1(r)}=\theta_1.
$ 
\end{proof}
\end{prop}

\medskip
In general, it is not true that $\theta_x\left(M^\times\right) \subseteq M^\times$, where $M^X$ is the group of units of a monoid $M$, as we show in the next example.
\begin{ex}\cite[Appendix B]{tesi} \label{ex5}
 Let us consider the set $X=\{1,a,b\}$ and the idempotent commutative monoid $M$ on $X$ with identity $1$ and such that $ab=b$. Clearly, $M^X=\{1\}$. Then, there exists an idempotent solution on $M$ for which $\theta_b(1)=a.$ Such a solution is defined considering the maps $\theta_1=\theta_a:M \to M$ given by $\theta_1(1)=\theta_a(1)=1$ and $\theta_1(b)=b$ and the map $\theta_b: M \to M$ given by $\theta_b(1)=\theta_b(a)=a$ and $\theta_b(b)=b$.
\end{ex}

\medskip

In the last part of this section, we will focus on idempotent solutions on monoids $M$ for which $\theta_1$ is also a homomorphism from $M$ to $\E(M)$. This assumption is not restrictive, as we show in the next result. Indeed, it is a necessary condition for idempotent solutions defined on monoids in which the idempotents are central. In the following, let us denote by $\Z(M)$ the center of $M$.

\begin{prop}\label{e_m_z_m}
Let $M$ be a monoid such that $\E(M) \subseteq \Z(M)$ and $s(x,y)=(xy, \theta_x(y))$ an idempotent solution on $M$. Then, the map $\theta_1$ is an idempotent monoid homomorphism from $M$ to $\E(M)$.
\begin{proof}
Initially, by \cref{prop_sidemp_monoide}-$4.$, the map $\theta_1$ is idempotent. Besides, recalling that $\theta_1(x) \in \E(M)$, for any $x \in M$, we have that
\begin{align*}
    \theta_1(xy)&=\theta_1(x)\theta_x(y)\theta_1(x) &\mbox{by \cref{cor_eSe}-$1.$}\\
    &=\theta_1\theta_1(x)\theta_{\theta_1(x)}(y)\theta_1(x) &\mbox{by  \cref{prop_sidemp_semigruppo}-$1.$}\\
    &=\theta_1\left(\theta_1(x)y\right)\theta_1(x) &\mbox{by \eqref{p_one}}\\
    &=\theta_1\left(y\theta_1(x)\right)\theta_1(x)\\
    &=\theta_1(y)\theta_y\theta_1(x) \theta_1(x) &\mbox{by \eqref{p_one}}\\
     &=\theta_1(y)\theta_1(x)\theta_y\theta_1(x)  \\
    &=\theta_1(y) \theta_1(x)\theta_{\theta_1(y)}\theta_1(x) &\mbox{by  \cref{prop_sidemp_semigruppo}-$1.$}\\
    &=\theta_1(y) \theta_1(x) &\mbox{by \eqref{i_one}}
\end{align*}
for all $x, y \in M$. Finally, by  \cref{prop_sidemp_monoide}-$1.$, it holds that $\theta_1(1)=1$, hence the claim follows.
\end{proof}
\end{prop}

    \noindent Note that the converse of \cref{e_m_z_m} is not true. Indeed, the map $s(x, y)=(xy, 1)$ is a solution in any monoid.

\medskip

\begin{lemma}\label{rem_I1}
Let $M$ be a monoid such that $\E(M) \subseteq \Z(M)$ and  $s(x,y)=(xy, \theta_x(y))$ a solution on $M$ for which $\theta_1$ is an idempotent monoid homomorphism from $M$ to $\E(M)$ such that $x=x\theta_1(x)$, for every $x \in M$. Then, $s$ is idempotent if and only if \eqref{i_two} is satisfied. 
\end{lemma}
\begin{proof}
It is enough to notice that \eqref{i_one} holds since, by \eqref{p_one},
\begin{align*}   xy&=x\theta_1(x)y\theta_1(y)=xy\theta_1(xy)=xy\theta_1(x)\theta_x(y)=x\theta_1(x)y\theta_x(y)=xy\theta_x(y),
\end{align*}
for all $x, y \in M$. 
\end{proof}

\medskip
The next result is a description of all idempotent solutions on a monoid $M$ having central idempotents. 
\begin{theor}\label{theor_solu_monoid}
 Let $M$ be a  monoid for which  $\E(M)  \subseteq \Z(M)$ and $\mu$ an idempotent monoid homomorphism from $M$ to $\E(M)$ such that, for every $x \in M$, $\mu(x)=e_x$, with $e_x \in \E(M)$ a right identity for $x$.
 Moreover, let $\{ \theta_e: M \to M \, \mid \, e \in \im \mu\}$ be a family of maps such that $\theta_1=\mu$, and for every $e \in \im \mu$,
 \begin{align}\label{ast}
     \theta_e(xy)=\theta_e(x)\theta_f(y),
 \end{align} 
 for all $x,y \in M$,  with $f=\mu(ex)$, and
 \begin{align}\label{astast}
     \theta_e=\theta_e\theta_{ef},
 \end{align}
 for every $ f \in \im \mu$, and
  \begin{align}\label{astastast}
     \theta_{ef}\theta_e(x)=\theta_e(x),
 \end{align}
 for every $x\in M$, with $f=\mu(x)$. Then, set $\theta_x=\theta_{\mu(x)}$, for every $x \in M$, one has that the map $s: M\times M \to M \times M$ given by  $s(x,y)=(xy, \theta_x(y))$ is an idempotent solution on $M$. Conversely, every idempotent solution on $M$ can be so constructed.
 \begin{proof}
 Let $x, y, z \in M$. Then, using \eqref{ast} we obtain \eqref{p_one}, since
 \begin{align*}
     \theta_x(y)\theta_{xy}(z)=\theta_{\mu(x)}(y)\theta_{\mu(\mu(x)y)}(z)=\theta_{\mu(x)}(yz)=\theta_x(yz).
 \end{align*}
Moreover,
 \begin{align*}
     \theta_{\theta_x(y)}\theta_{xy}(z)&=\theta_{\mu\theta_x(y)}\theta_{\mu(x)\mu(y)}(z) \\
     &=\theta_{\theta_1\theta_x(y)}\theta_{\theta_1(x)\theta_1(y)}(z)\\
     &=\theta_{\theta_1(y)}\theta_{\theta_1(x)\theta_1(y)}(z) &\mbox{by \eqref{astast} since $\theta_1=\theta_1\theta_x$}\\
     &=\theta_y(z)   &\mbox{by \eqref{astast}}
 \end{align*}
and so \eqref{p_two} is satisfied.  Thus, by \cref{rem_I1}, \eqref{i_one} holds.  Finally, applying \eqref{astastast}, we get
\begin{align*}
    \theta_{xy}\theta_x(y)=\theta_{\mu(x)\mu(y)}\theta_{\mu(x)}(y)=\theta_{\mu(x)}(y)=\theta_x(y).
\end{align*}
Therefore, $s(x,y)=(xy, \theta_x(y))$ is an idempotent solution on $M$.\\
Vice versa, if we assume that $s(x,y)=(xy, \theta_x(y))$ is an idempotent solution on $M$, then by \cref{e_m_z_m}, $\mu=\theta_1$ is an idempotent monoid homomorphism, and, by \cref{prop_Se}-$2.$(e), we have that $x \in M \theta_1(x)$, for every $x \in M$. In addition, by \cref{prop_sidemp_monoide}-$2.$, $\theta_x=\theta_{\theta_1(x)}$, for every $x \in M$. Hence, by\eqref{p_one}, we obtain \eqref{ast}. Now, let $e,f \in \im\theta_1$, thus there exist $x,y \in M$ such that $e=\theta_1(x)$ and $f=\theta_1(y)$.  Besides, by \eqref{p_two}, $$\theta_e=\theta_{\theta_1(x)}=\theta_x=\theta_{\theta_y(x)}\theta_{yx}=\theta_{\theta_1(x)}\theta_{\theta_1(y)\theta_1(x)}=\theta_e\theta_{fe},$$
and so \eqref{astastast}  holds. Finally, by \eqref{i_two}, if $e \in \im \theta_1$, $x \in M$, and $f=\theta_1(x)$, we obtain
$$\theta_{ef}\theta_e(x)=\theta_{\theta_1(xy)}\theta_{\theta_1(x)}(y)=\theta_{xy}\theta_x(y)=\theta_x(y)=\theta_{\theta_1(x)}(y)=\theta_e(y),$$
for every $y \in M$, hence \eqref{astastast} holds.
 \end{proof}
\end{theor}

\medskip

\begin{rem}
    Unlike solutions defined on groups, in the case of idempotent solutions on monoids, it is not possible to find a way to write the maps $\theta_x$ by means of the map $\theta_1$ as in \cref{teo_gruppi}. Indeed, if we look at the idempotent solutions on the monoid $M$ in \cref{ex5}, one can see that there are three different solutions having the same map $\theta_1:M \to M$ given by $\theta_1(1)=\theta_a(1)=1$ and $\theta_1(b)=b$.
\end{rem}

\medskip

The next result narrows down that choice of the maps $\theta_e$ in \cref{theor_solu_monoid}. Indeed, given an idempotent solution on a monoid $M$ such that $\theta_1$ is a monoid homomorphism from $M$ to $\E(M)$, one has that the kernel of $\theta_1$, i.e., the set
\begin{align*}
    \ker \theta_1= \{(x,y) \in M \times M \, \mid \, \theta_1(x)=\theta_1(y)\}
\end{align*}
is a congruence relation on $M$. Thus, one can naturally consider the quotient monoid $M/\ker \theta_1$ (see \cite{Rhod89} for more details). Additionally, we have the following properties.

\begin{theor}\label{theor_kertheta1}
 Let $M$ be a monoid and $s(x,y)=(xy, \theta_x(y))$ an idempotent solution on $M$ such that $\theta_1$ is a monoid homomorphism from $M$ to $\E(M)$.  Then, 
 \begin{enumerate}
 \item $\theta_1(M)$ is a system of representatives of $M/\ker \theta_1$ that contains the identity;
     \item $(\theta_x(y) ,y) \in \ker \theta_1 $, for all $x, y \in M$.
 \end{enumerate}
\begin{proof}
The first part is a consequence of the idempotence of the map $\theta_1$ by $4.$ in \cref{prop_sidemp_monoide}. In fact,  \cref{prop_sidemp_monoide}-$1.$, we have that $1=\theta_1(1) \in \theta_1(M)$. Moreover, since by \cref{prop_sidemp_monoide}-$4.$ $\theta_1$ is an idempotent map, we easily obtain that $(\theta_1(x),x) \in \ker \theta_1 $, for every $x \in M$. Besides, if $x,y \in M$ are distinct and such that $(\theta_1(y) ,x) \in \ker \theta_1$, then we get $\theta_1(y)=\theta_1(\theta_1(y))=\theta_1(x)$.\\
The second part follows by \cref{cor_eSe}-$3.$, since  $\theta_1\theta_x(y)=\theta_1(y)$, for all $x, y \in M$, and so $(\theta_x(y) ,y) \in \ker \theta_1$. Therefore, we get the claim.
\end{proof}
\end{theor}

\vspace{3mm}

\begin{cor}
    Let $M$ be a monoid for which $\E(M) \subseteq \Z(M)$ and $s(x, y)=(xy, \theta_x(y))$ an idempotent solution on $M$. Then, $\theta_1(M)$ is a system of representatives of $M/\ker \theta_1$ that contains the identity and $\left(\theta_x(y), y\right) \in \ker \theta_1$.
    \begin{proof}
       It is a consequence of \cref{theor_kertheta1}, since, by  \cref{e_m_z_m}, the map $\theta_1$ is a monoid homomorphism from $M$ to $\E(M)$.
    \end{proof}
\end{cor}

\medskip

\begin{rem}
Let $M$ be a monoid for which $\E(M) \subseteq \Z(M)$ and $s(x,y)=(xy, \theta_x(y))$ and $r(x, y)=\left(xy, \eta_x(y)\right)$ two idempotent solutions on $M$. Then, by \eqref{isomor},  if such solutions are isomorphic, there exists an isomorphism $f$ of $M$ such that $f\theta_1=\eta_1 f$, i.e., $f$ sends the system of representatives $\theta_1(M)$ of $M/\ker \theta_1$ into the other one $\eta_1f(M)$.
    
\end{rem}

\bigskip 
\bibliography{bibliography}

\end{document}